\documentclass[12pt,reqno]{amsart}
\usepackage[utf8]{inputenc}
\usepackage{amsmath}

\usepackage[margin=1in]{geometry}  
\usepackage{amsmath}               
\usepackage{amsfonts}              
\usepackage{amsthm}      
\usepackage{enumerate}
\usepackage{amssymb}
\usepackage{bbm}

\usepackage[english]{babel}

\usepackage[T1]{fontenc}
\usepackage[utf8]{inputenc}
\usepackage{mathtools}

\usepackage{xcolor}   
\usepackage{hyperref}
\hypersetup{
    colorlinks=true, 
    linktoc=all,     
    linkcolor=black,  
}

\DeclarePairedDelimiter\abs{\lvert}{\rvert}%
\DeclarePairedDelimiter\norm{\lVert}{\rVert}%

\makeatletter
\let\oldabs\abs
\def\abs{\@ifstar{\oldabs}{\oldabs*}}
\let\oldnorm\norm
\def\norm{\@ifstar{\oldnorm}{\oldnorm*}}

\makeatother

\theoremstyle{definition}
\newtheorem{theorem}{Theorem}[section]
\newtheorem{lemma}[theorem]{Lemma}
\newtheorem{proposition}[theorem]{Proposition}

\newtheorem{corollary}[theorem]{Corollary}

\newtheorem{definition}[theorem]{Definition}
\newtheorem{example}[theorem]{Example}

\newtheorem{remark}[theorem]{Remark}

\newtheorem{question}{Question}
\numberwithin{equation}{section}

\DeclareMathOperator{\Cay}{Cay}

\DeclareMathOperator{\Prob}{Prob}

\DeclareMathOperator{\clarep}{ClaRep}

\let\phi\varphi

\let\empt\varnothing
 
\newcommand{\eps}{\varepsilon}

\newcommand{\acts}{\curvearrowright}
\newcommand{\actson}{\curvearrowright}

\newcommand{\EE}{\mathbb{E}} 

\newcommand{\RR}{\mathbb{R}}
\newcommand{\HH}{\mathbb{H}}

\newcommand{\PP}{\mathbb{P}}

\newcommand{\NN}{\mathbb{N}}     

\newcommand{\one}{\mathbbm{1}}

\newcommand\restr[2]{{
  \left.\kern-\nulldelimiterspace 
  #1 
  \vphantom{\big|} 
  \right|_{#2} 
  }}

\newcommand{\cal}[1]{{\mathcal #1}}

\newcommand\GG{\mathbb G_\bullet}
\newcommand\GGG{\mathbb G_{\bullet\bullet}}

\DeclareMathOperator{\IRE}{IRE}
\DeclareMathOperator{\FIRE}{FIRE}
\DeclareMathOperator{\AIRE}{AIRE}

\DeclareMathOperator{\Rre}{\mathcal{R}_\textrm{re}}

\DeclareMathOperator{\BVT}{BVT}

\NewDocumentCommand{\expect}{ e{^} s o >{\SplitArgument{1}{|}}m }{%
  \operatorname{\mathbb{E}}
  \IfValueT{#1}{{\!}^{#1}}
  \IfBooleanTF{#2}{
    \expectarg*{\expectvar#4}%
  }{
    \IfNoValueTF{#3}{
      \expectarg{\expectvar#4}%
    }{
      \expectarg[#3]{\expectvar#4}%
    }%
  }%
}
\NewDocumentCommand{\expectvar}{mm}{%
  #1\IfValueT{#2}{\nonscript\;\delimsize\vert\nonscript\;#2}%
}
\DeclarePairedDelimiterX{\expectarg}[1]{[}{]}{#1}

\makeatletter
\def\l@section{\@tocline{1}{2pt}{0pc}{}{}}
\renewcommand{\tocsection}[3]{%
	\indentlabel{\@ifnotempty{#2}{\makebox[20pt][l]{%
				\ignorespaces#1 #2.\hfill}}}\sc #3\dotfill}

\newdimen{\tocsubsecmarg}
\def\l@subsection{\@tocline{2}{0pt}{0pc}{\tocsubsecmarg}{}}
\renewcommand{\tocsubsection}[3]{%
	\indentlabel{\@ifnotempty{#2}{\makebox[30pt][l]{%
				\ignorespaces#1 #2.\hfill}}}#3\dotfill}
\makeatother
\let\oldtocsubsection=\tocsubsection
\renewcommand{\tocsubsection}[2]{\hspace{2em} \oldtocsubsection{#1}{#2}}

\title[Exactness and the topology of the space of IREs]{Exactness and the topology of the space of invariant random equivalence relations}
\author[H\'ector Jard\'on-S\'anchez]{H\'ector Jard\'on-S\'anchez}
\address[H\'ector Jard\'on-S\'anchez]{Department of Mathematics, Jagiellonian University, Krak\'ow, Poland}
\email{hectorjardon@gmail.com}
\author[Sam Mellick]{Sam Mellick}
\address[Sam Mellick]{Department of Mathematics, Jagiellonian University, Krak\'ow, Poland}
\email{samuel.mellick@uj.edu.pl}
\author[Antoine Poulin]{Antoine Poulin}
\address[Antoine Poulin]{Department of Mathematics, McGill University, Montreal, Canada}
\email{antoine.poulin@mail.mcgill.ca}
\author[Konrad Wr\'{o}bel]{Konrad Wr\'{o}bel}
\address[Konrad Wr\'{o}bel]{Department of Mathematics, The University of Texas at Austin, USA}
\email{konrad.wrobel@austin.utexas.edu}

  
\begin{document}

\begin{abstract}
  We characterize exactness of a countable group $\Gamma$ in terms of invariant random equivalence relations (IREs) on $\Gamma$. Specifically, we show that $\Gamma$ is exact if and only if every weak limit of finite IREs is an amenable IRE.  In particular, for exact groups this implies amenability of the restricted rerooting relation associated to the ideal Bernoulli Voronoi tessellation, the discrete analog of the ideal Poisson Voronoi tessellation. 
\end{abstract}

\maketitle
\vspace{-1ex}
\tableofcontents


\section{Introduction}

Let $\Gamma$ be a countable discrete group and let $\cal E \subseteq 2^{\Gamma\times \Gamma}$ denote the standard Borel space of equivalence relations on $\Gamma$. An \textbf{invariant random equivalence relation (IRE)} is a Borel probability measure $\rho$ on $\cal E$ which is invariant under the left shift action $\Gamma \actson \cal E$.  
A \textbf{finite IRE} is an IRE $\rho$ such that $\rho$-almost every equivalence relation has only finite classes. Invariant random equivalence relations (IREs) on groups were first introduced by Tucker-Drob \cite{TDthesis} and later studied by Kechris \cite{KechrisSpaceOfEqRels}. 

In this paper, we primarily concern ourselves with amenability of IREs. An \textbf{amenable IRE} is an IRE $\rho$ for which there exist Borel maps $\lambda_n\colon \cal E \to \Prob (\Gamma)$ such that for $\rho$-almost every $r\in \cal E$ and for every $g\in [e]_r$,
    \[
    \|g^{-1}.\lambda_n^r - \lambda_n^{g^{-1}.r}\|_1 \longrightarrow 0.
    \] 
For example, the \textit{full} IRE $\delta_{\Gamma\times\Gamma}$ is amenable if and only if $\Gamma$ is amenable. 

An immediate consequence of a result of Kechris \cite[Theorem 15.7]{KechrisSpaceOfEqRels} is that the full IRE on a group is amenable if and only if it is a weak limit of finite IREs. In fact, every amenable IRE is a weak limit of finite IREs (see Proposition \ref{prop:AIREsarewFIREs}). Conversely, one might also expect that weak limits of finite IREs are amenable in general. Perhaps surprisingly, this expectation fails for every nonexact group.

\begin{theorem}[see Theorem \ref{thm:3}]\label{thm:introNonexact}
    Let $\Gamma$ be a countable discrete group. If $\Gamma$ is nonexact, then there exists a weak limit of finite IREs on $\Gamma$ which is not amenable. 
\end{theorem}

Exactness of groups is a weakening of amenability, introduced independently by Kirchberg and Wassermann \cite{kw} and by Yu \cite{yu}, who moreover proved that these groups satisfy the coarse Baum-Connes conjecture. A group $\Gamma$ is \textbf{exact} if and only if it admits a continuous action on a compact space which is topologically amenable \cite{anan,GJ,HR,ozawa}. 
For a finitely generated group, exactness is equivalent to uniform local amenability  of its Cayley graphs \cite{ULA,ElekUla}. 
Therefore bounded degree Cayley graphs of nonexact groups admit an embedded sequence of small scale expanders (see Definition \ref{def:smalscexp}). 
We prove Theorem \ref{thm:introNonexact} by constructing finite IREs on $\Gamma$ whose classes are close enough to small scale expanders so that the weak limit cannot be amenable.  

On the other hand, under the assumption of exactness, the initial expectation about amenability of weak limits of finite IREs is realized. 

\begin{theorem}[see Theorem \ref{thm:exactWFIREisAIRE}]\label{thm:introExact}
    Let $\Gamma$ be a countable discrete group. If $\Gamma$ is exact, then every weak limit of  finite IREs on $\Gamma$ is amenable. 
\end{theorem}

A topologically amenable action on a compact space can be thought of as a boundary of $\Gamma$. 
The key idea in the proof of Theorem \ref{thm:introExact} is that any weak limit $\rho$ of finite IREs on an exact group admits a marking by boundary points which is equivariant on classes. 
The existence of such a marking then implies that $\rho$ is an amenable IRE.

Let $\IRE_\Gamma$, $\AIRE_\Gamma$, and $\FIRE_\Gamma$ denote the spaces of all IREs, amenable IREs, and finite IREs on $\Gamma$ respectively, endowed with the weak convergence topology. Since for every group $\Gamma$ every amenable IRE on $\Gamma$ is a weak limit of finite IREs, Theorems \ref{thm:introNonexact} and   \ref{thm:introExact}  give a characterization of exactness in terms of IREs, \`a la Kechris's characterization \cite[Theorem 15.7]{KechrisSpaceOfEqRels} of amenability being equivalent to $\overline{\FIRE}_\Gamma = \IRE_\Gamma$. 

\begin{corollary}
    A countable discrete group is exact if and only if $\overline{\FIRE}_\Gamma=\AIRE_\Gamma$. 
\end{corollary}

\medskip\noindent
\textbf{Additional motivation.}
 The original impetus for this work was to better understand the proof of fixed price one for higher rank semisimple Lie groups \cite{IPVTHighRank}. There, the main object of study is the ideal Poisson Voronoi tessellation (IPVT) of a locally compact second countable group $G$. The IPVT is a weak limit of IREs with finite volume classes -- it arises as the low-intensity limit of the Voronoi tessellations associated to Poisson point processes on $G$. A key step in \cite{IPVTHighRank} is proving an amenability property of this IRE. It is a natural question to ask whether amenability is a general consequence of being a limit of IREs with finite volume classes. In Theorems \ref{thm:introNonexact} and \ref{thm:introExact} we fully resolve the discrete version of this question.  
 
 The first studied examples of the IPVT may be found in the PhD thesis of Bhupatiraju \cite{bhu}, where the model is studied on regular trees and the hyperbolic plane. The concept also appears in \cite{BCP}. An explicit description of the IPVT in hyperbolic spaces was shown in \cite{russetal}, and independently for symmetric spaces in \cite{IPVTHighRank}. Recently, in \cite{dachille} the IPVT on $\HH^2 \times \HH^2$ with the $\ell_1$ metric was given an explicit description. 

In Section \ref{sec:IBVTQ} we pose some open questions on the discrete analog of the IPVT: the ideal Bernoulli Voronoi tessellation. 

\medskip\noindent
\textbf{Acknowledgements.} The authors would like to warmly thank Miko\l aj Frączyk for suggesting the construction that appears in Theorem \ref{thm:3}, which was in turn inspired by an observation of Tom Hutchcroft. The authors are sincerely grateful to Anush Tserunyan for several fruitful remarks and suggestions. The authors would also like to thank {\L}ukasz Grabowski and Forte Shinko for various other helpful conversations. HJS, SM, and KW were supported by the Dioscuri program initiated by the Max Planck Society, jointly managed by the National Science Centre (Poland), and mutually funded by the Polish Ministry of Science and Higher Education and the German Federal Ministry of Education and Research.

\section{Preliminaries}\label{sec:prelims}

Throughout, $\Gamma$ will denote a countable discrete group. If $\Gamma \actson X$ is an action on a set $X$, we denote the action of $g \in \Gamma$ on an element $x\in X$ by $g.x$. A \textbf{probability measure preserving (pmp)} action $\Gamma \actson (X,\mu)$ is an action on a standard probability space $(X,\mu)$ satisfying $\mu(A)=\mu(g.A)$ for every measurable $A\subseteq X$ and $g\in\Gamma$. 

\subsection{Exact groups}

The reader is referred to \cite[Section 4.2 and Chapter 5]{Bro} for an introduction to topologically amenable actions and exact groups. We review the definitions necessary for our purposes here.

Let $\Prob (\Gamma)$ denote the set of probability measures on $\Gamma$ equipped with the topology inherited as a subspace of $\ell^1(\Gamma)$. Let $\Gamma \actson \Prob(\Gamma)$ be the action given by the shift ${g.\eta(h) = \eta(g^{-1}h)}$, where $g, h \in \Gamma$ and $\eta \in \Prob(\Gamma)$. 

\begin{definition}
An action of a countable group $\Gamma$ on a topological space $X$ by homeomorphisms is \textbf{topologically amenable} if it admits a sequence $(\eta_n)_{n\in \NN}$ of continuous maps $\eta_n\colon X\to \Prob(\Gamma)$ such that for each $g\in \Gamma$,
\[
\sup _{x\in X} \|g.\eta_n^x-\eta_n^{g.x}\|_1\longrightarrow 0.
\]
A countable group is \textbf{exact} if it admits a topologically amenable action on a compact metrizable space.
\end{definition}

This is not one of the standard definitions of exactness as we include the additional assumption of metrizability. However, it is equivalent for countable groups by \cite[Theorem 5.1.7]{Bro}. This characterization of exactness is often referred to as \emph{boundary amenability} or \emph{amenability at infinity}.
Many groups are exact, for instance all linear groups \cite{GHW} and hyperbolic groups \cite{AdamsRelHyp}.
Examples of nonexact groups have been produced by Gromov \cite{Grom:SpacesAndQuestions,Grom:RWRG} and Osajda \cite{osajda}. For more on exactness, see \cite{anansurvey}.

In Section \ref{sec:positive} we will use the above definition of exactness in terms of topologically amenable actions, whilst in Section \ref{sec:counter} we will make use of the combinatorial characterization of nonexactness given below.

\begin{definition}\label{def:smalscexp}
A finite graph $G$ is a \textbf{small scale $(\kappa,N)$-expander}, for some $\kappa> 0$ and $N \in \NN$, if for every $F \subseteq V(G)$ with $|F| \leq N$ we have $$
\frac{|\partial_G F|}{|F|} \geq \kappa.
$$
A sequence of finite graphs $(G_n)_{n\in \NN}$ is a \textbf{sequence of small scale $\kappa$-expanders}, for some $\kappa >0$, if each $G_n$ is a small scale $(\kappa,n)$-expander. We say that $(G_n)_{n\in \NN}$ is a sequence of \textbf{small scale expanders} if there exists some $\kappa > 0$ for which it is a sequence of small scale $\kappa$-expanders.
\end{definition}

\begin{theorem}[Brodzki--Niblo--\v{S}pakula--Willett--Wright \cite{ULA}, Elek \cite{ElekUla}]
    Let $\Gamma$ be a finitely generated group with finite generating set $S$. Then $\Gamma$ is nonexact if and only if the Cayley graph $\Cay (\Gamma , S)$ contains a sequence of small scale expanders as induced subgraphs.
\end{theorem}

\subsection{Borel and measured combinatorics}

Let $X$ be a standard Borel space. Recall that a \textbf{countable Borel equivalence relation (cber)} on $X$ is a Borel subset $\cal R \subseteq X\times X$ defining an equivalence relation on $X$ with countable equivalence classes. If the relation has finite classes, we say it is a finite Borel equivalence relation. 

A \textbf{probability measure preserving (pmp) cber} on a standard probability space $(X,\mu)$ is a cber $\cal R$ such that $\int c_x\;d\mu=\int c^x\;d\mu$ where $c_x$ and $c^x$ are the counting measures on $\{x\}\times[x]_{\cal R}$ and $[x]_{\cal R} \times\{x\}$, respectively. 

\begin{definition}
    A cber $\cal R$ on $X$ is \textbf{Borel amenable} if there exists a sequence of \textbf{Hulanicki--Reiter functions}: a sequence of Borel maps $\eta_n \colon \cal R\to \RR_{\geq 0}$ such that
    \begin{itemize}
        \item[i)] $\eta_n^x \in \Prob ([x]_{\cal R})$ for all $x\in X$ and $n\in \NN$, where $\eta_n^x (y) \coloneqq \eta_n (y,x)$ for every $y\in [x]_{\cal R}$, and
        \item[ii)] $\|\eta_n^x - \eta_n^y \|_1 \longrightarrow 0$ for every $(y,x) \in \cal R$.
    \end{itemize}

\end{definition}

A pmp cber $\cal R$ on a standard probability space $(X,\mu)$ is \textbf{$\mu$-amenable} if its restriction to a $\mu$-conull $\cal R$-invariant Borel subset of $X$ is Borel amenable. For more on amenability of equivalence relations, the reader is referred to \cite{KM} and \cite{Moore:AmenableRelations}.

If $\Gamma\actson X$ is a Borel action of a countable group $\Gamma$ on a standard Borel space $X$, we let $\cal R_{\Gamma\actson X}$ denote the \textbf{orbit equivalence relation}. This cber is defined by letting $(y,x)\in \cal R_{\Gamma\actson X}$ if and only if there exists $g \in \Gamma$ such that $g.x = y$, for any $x,y\in X$.

\begin{lemma}\label{lem:BorelAmenableRetract}
    Let $\Gamma \acts X$ be a Borel action on a standard Borel space and ${\cal S \subseteq \cal R_{ \Gamma \actson X}}$ a Borel subequivalence relation. Suppose there exists a sequence  of Borel maps ${\eta_n\colon X \rightarrow \Prob(\Gamma)}$, for $n\in\NN$, such that for every $ x\in X$ and $g \in \Gamma$ satisfying $(g.x,x)\in \cal S$,
    \[ \|g.\eta_n^x - \eta_n^{g.x} \|_1 \longrightarrow 0.\]
    
    Then $\cal S$ is Borel amenable.
\end{lemma}
\begin{proof}   
Using Lusin-Novikov uniformization and a standard argument, we may construct a Borel ``class representative'' map $\clarep \colon X \to X^\NN$ such that, for each $y \in X$, the sequence $\clarep (y)$ contains a unique representative from each $\mathcal S$-class in $[y]_{\mathcal R_{\Gamma \actson X}}$. That is, for each $y \in X$ and $x \in [y]_{\mathcal R_{\Gamma \actson X}}$ there exists a unique $i(x;y) \in \NN$ such that $\clarep(y)_{i(x;y)} \in [x]_{\mathcal S}$. We assume that $\clarep(y)_1 = y$ for every $y \in X$.

For each $x \in X$, we define a retraction of its $\mathcal R_{\Gamma \actson X}$-class onto its $\mathcal S$-class $\phi_x \colon [x]_{\mathcal R_{{\Gamma \actson X}}} \to [x]_{\mathcal S}$ by letting, for each $y \in [x]_{\mathcal R_{\Gamma \actson X}}$, $$\phi_x (y) = \clarep(y)_{i(x;y)}.
$$
The map $\phi_x$ is $\mathcal S$-invariant  in the sense that for every $z \in [x]_{\mathcal S}$ and $y \in [x]_{\mathcal R_{{\Gamma \actson X}}}$ we have $\phi_x(y) = \phi_z(y)$. 

We define witnesses $\xi_n \colon \mathcal S \to \RR_{\geq 0}$ to Borel amenability as follows. For each $x \in X$ and $y \in [x]_{\mathcal S}$ we let $$
\xi_n^x (y) := \sum_g \eta_n^x (g^{-1}),
$$
where the sum ranges over all $g \in \Gamma$ such that $\phi_x (g.x) = y$. On the one hand, for every $x \in X$ we have
\[
\sum_{y \in [x]_{\mathcal S}} \xi_n^x (y) = \sum_{g \in \Gamma} \eta^x_n (g^{-1}) = \sum_{g \in \Gamma} \eta^x_n (g) = 1.
\]
On the other hand, for any $(g.x,x) \in \mathcal S$, by the triangle inequality and $\mathcal S$-invariance of $\phi$, \begin{align*}
\|\xi_n^x - \xi_n^{g.x}\|_1 &= \sum_{z \in [x]_{\mathcal S}} |\xi_n^x (z) - \xi_n^{g.x} (z)|\\ &\leq \sum_{h \in \Gamma} |\eta_n^x (g^{-1}h^{-1}) - \eta_n^{g.x} (h^{-1})|\\
& = \|g.\eta_n^x - \eta_n^{g.x}\|_1\to 0.
\end{align*}
It follows that $\cal S$ is Borel amenable. 
\end{proof}

A \textbf{Borel graph} on a standard Borel space $X$ is a symmetric Borel subset $\cal G \subseteq X\times X$. Say the graph has \textbf{bounded degree} if there exists $d\in\NN$ such that $\deg_{\cal G}(x)\leq d$ for every $x\in X$. For a vertex $x\in X$, let $\cal G_x$ denote its connected component in $\cal G$. A bounded degree Borel graph $\cal G$ induces a cber $\cal R _{\cal G}$ on $X$ defined by taking the vertices of each connected component to be their own equivalence class, i.e.,  $x,y\in X$ are $\cal R_{\cal G}$-equivalent if $\cal G_x = \cal G_y$.

Let $(X,\mu)$ be a standard probability space. A Borel graph $\cal G$ on $X$ is \textbf{probability measure preserving (pmp)} if $\cal R_{\cal G}$ is a pmp cber on $(X,\mu)$. If $\cal R$ is a pmp cber and $\cal G$ a Borel graph on $(X,\mu)$, we say that $\cal G$ is a \textbf{graphing} of $\cal R$ if $\cal R_{\cal G} = \cal R$ almost everywhere, i.e., for $\mu$-almost every $x \in X$, we have $[x]_{\mathcal R} = [x]_{\mathcal R_{\mathcal G}}$.

\subsection{Graph convergence}

In this subsection we recall some basic facts about Benjamini--Schramm convergence. The reader is referred to the survey paper of Aldous and Lyons \cite{AldousLyons} and the book of Lovasz \cite{Lovasz} for a thorough introduction to the topic. 

Let $\GG$ denote the set of isomorphism classes of rooted, connected, locally finite graphs. For a rooted graph $(G,u)$, let $B_r(G,u)$ denote the rooted ball of radius $r$ in $G$ about $u$. We endow $\GG$ with the complete and separable metric defined by 
\[
    d((G,u),(H,v)) = \inf\{2^{-r} : B_r (G,u) \cong B_r (H,v)\},
\]
 for $(G,u),(H,v)\in\GG$.
In a similar way we also define $\GGG$, the space of doubly rooted, connected, locally finite graphs. An element $(G, u, v)$ of $\GGG$ is a graph with two distinguished vertices $u, v \in G$ (possibly with $u = v$).

A \textbf{unimodular random graph} is a random rooted graph $(G, o)$ of $\GG$ that satisfies the Mass Transport Principle: for every Borel map $f\colon \GGG \to \RR_{\geq 0}$, 
\[
\EE\left[\sum_{v\in V(G)} f(G,v,o)\right] = \EE \left[\sum_{v\in V(G)} f(G,o,v)\right].
\]

\begin{example}[Cayley graphs]
    Let $\Gamma$ be a finitely generated group with generating set $S$. An example of a unimodular random graph is the Cayley graph $\Cay(\Gamma, S)$ rooted at the identity $e\in \Gamma$, where we treat $\Cay(\Gamma,S)$ as the random rooted graph with distribution the Dirac $\delta_{\Cay(\Gamma,S)}$. In this specific case, the mass transport principle takes the following form. If $F : \Gamma \times \Gamma \to \RR_{\geq 0}$ is diagonally invariant, then
    \[
        \sum_{g \in \Gamma} F(e, g) = \sum_{g \in \Gamma} F(g, e).
    \]   
    We will apply this to functions $F$ of the form $F(g, h) = \EE[f(g, h; \Pi)]$, where $\Pi \in X$ is a $\Gamma$-invariant stochastic process and $f: \Gamma\times\Gamma\times X \rightarrow \RR_{\geq 0}$ is diagonally invariant (that is, $f(g h, g h'; g.\Pi) = f(h, h'; \Pi)$). 
\end{example}

\begin{example}
    Let $\cal G$ on $(X,\mu)$ be a pmp bounded degree Borel graph. Then the random variable $x\in X \mapsto (\cal G_x,x)\in \GG$ is a unimodular random graph.
\end{example}

\begin{example}
    A particular instance of the latter construction is the unimodular random graph associated to a finite graph $G$, defined as $(G,v) \in \GG$, where $v\in V(G)$ is sampled from the uniform distribution on $V(G)$. 
\end{example}

    A sequence of unimodular random graphs $(G_n,o)$ \textbf{Benjamini--Schramm converges} to a unimodular random graph $(G,o)$ if $B_r (G_n,o)$ converges in distribution to $B_r (G,o)$ for every $r\in \NN$.

\subsection{Hyperfiniteness}

In this subsection we discuss hyperfiniteness of cbers and unimodular random graphs. 

A cber $\cal R$ on $X$ is \textbf{hyperfinite} if there exists an increasing union of finite Borel subequivalence relations $(\cal F_n)_n$ of $\cal R$ such that $\cal R = \bigcup_n \cal F_n$.
A pmp cber $\cal R$ on $(X,\mu)$ is $\mu$\textbf{-hyperfinite} if the restriction of $\cal R$ to a $\mu$-conull $\cal R$-invariant Borel subset is hyperfinite.
The classic article of Connes--Feldman--Weiss \cite{CFW} (building on the landmark work of Ornstein-Weiss\cite{OW80}) proves that $\mu$-amenability and $\mu$-hyperfiniteness are equivalent. The analogous purely Borel statement remains an important open question. In this paper, we will use the $\mu$-amenable perspective in Section \ref{sec:positive} and the $\mu$-hyperfinite perspective in Section \ref{sec:counter}.

A unimodular random graph $(G,o)$ is $(\eps,k)$\textbf{-hyperfinite} for $\eps>0$ and $k\in \NN$ if there exists a coupling $(G,E,o)$ with $E\subseteq E(G)$ such that \begin{enumerate}
    \item $\EE\big[\mathrm{deg}_E(o)\big] \leq \eps$ and
    \item $G- E$ has connected components of size at most $k$ almost surely. 
\end{enumerate}
Say that a unimodular random graph $(G,o)$ is \textbf{hyperfinite} if for every $\eps > 0$ there exists $k\in \NN$ such that $(G,o)$ is $(\eps,k)$-hyperfinite.

We will make use of the following well-known connection between these notions which follows immediately from \cite[Proposition 2.2]{Elek:FinGraphsAmen} and \cite[Theorem 1.1]{CGMTD}.

\begin{proposition}\label{prop:hyper}
    Let $\cal R$ be a pmp cber on $(X,\mu)$ and $\cal G$ a bounded degree graphing of $\cal R$. Then $\cal R$ is $\mu$-hyperfinite if and only if the unimodular random graph $(\cal G_x,x)$ is hyperfinite.
\end{proposition}

Schramm proved the following witnessing hyperfiniteness of a limit in the converging sequence.

\begin{theorem}[Schramm {\cite[Theorem 1.2]{schramm}}]\label{prop:schramm}
    Let $(G_n,o)$ be a sequence of unimodular random graphs Benjamini--Schramm converging to $(G,o)$. If $(G,o)$ is hyperfinite, then for every $\eps > 0$ there exists $k\in \NN$ such that for $n$ large enough $(G_n,o)$ is $(\eps,k)$-hyperfinite.
\end{theorem}

Finally, it will be useful to study (non-random) finite graphs. In that context, we say that a finite graph $G$ is $(\eps,k)$\textbf{-hyperfinite} if there exists a subset of edges $E \subseteq E(G)$ with $|E|\leq \eps |V(G)|$ such that $G\backslash E$ has connected components of size at most $k$. Hyperfiniteness of $G$ in this sense is equivalent to hyperfiniteness of the associated unimodular random graph $(G,v)$ \cite[Lemma 1.3]{schramm}.

\section{Invariant random equivalence relations}\label{sec:IRE}

Let $\Gamma$ be a countable group. Denote by $\cal E_\Gamma $ the set of equivalence relations on $\Gamma$ and equip it with the topology inherited as a subspace of $\{0,1\}^{\Gamma \times \Gamma}$. The group $\Gamma$ acts on $\cal E_\Gamma$ by the left shift, i.e., $g . r (h,h') = r (g^{-1}h, g^{-1}h')$
for $g, h,h' \in \Gamma$. 
We will often simply write $\cal E$, instead of $\cal E_\Gamma$.

\begin{definition}
    An \textbf{invariant random equivalence relation (IRE)} is a Borel probability measure $\rho$ on $\cal E$ which is invariant under the action $\Gamma \actson \cal E$. 
    \end{definition}

The space $\Prob (\cal E)$ of Borel probability measures on $\cal{E}$ is a compact metrizable space when endowed with the weak topology, hence in particular sequentially compact. The subspace $\IRE_\Gamma \subseteq \Prob (\cal E)$ of IREs is closed, so it is also sequentially compact. 
IREs have previously been studied by Tucker-Drob \cite{TDthesis} and Kechris \cite{KechrisSpaceOfEqRels}.

\begin{remark}[IREs as random variables]

Although we defined IREs as measures on the space $\mathcal E$, we will often abuse terminology by saying that a random variable $R \colon (\Omega, \PP) \to \mathcal E$ is an IRE if its \textbf{distribution} $\rho:= R_\ast (\PP)$ is an IRE, where $(\Omega,\PP)$ is a standard probability space. When using this notation, if $A \subset \mathcal E$ is Borel, we write $$
\PP [R \in A]:= \PP[\{\omega \in \Omega \mid R(\omega) \in A \}]=\rho(A).
$$
This abuse of terminology follows similar conventions about invariant random subgroups (IRS) in the literature (see for instance \cite{AbertGlasnerVirag,sevenSamurai}). Since $\cal E$ is compact metrizable, convergence in distribution of a sequence $R_n$ is equivalent to weak convergence of $(R_n)_*\PP$ in $\Prob(\cal E)$.

Section \ref{sec:positive} is heavily motivated by the theory of amenable equivalence relations (and more generally groupoids).  Treating IREs as measures yields proofs aligned with the literature on measured equivalence relations. 

On the other hand, Section \ref{sec:counter} relies on Benjamini--Schramm convergence and random tessellations which have natural descriptions as stochastic processes. Regarding IREs as random variables here allows us to reduce our analyses to standard methods in probability theory.
\end{remark}

\begin{example}[Cosets]
    Suppose $\Lambda \leq \Gamma$ is a subgroup. The left coset decomposition of $\Gamma$ defines an invariant equivalence relation (as $x\Lambda = y\Lambda$ if and only if $gx\Lambda = gy\Lambda$), and hence the Dirac mass on the left coset decomposition defines an IRE. Note that the two extreme cases of $\Lambda = \Gamma$ and $\Lambda = 1$ correspond to the full IRE $\delta_{\Gamma\times \Gamma}$ and the equality IRE $\delta_\Delta$ respectively, where $\Delta \subset \Gamma \times \Gamma$ denotes the diagonal. 

    One can also consider the \emph{right} coset decomposition. This defines an equivariant map from the space of subgroups of $\Gamma$ to the space of equivalence relations on $\Gamma$. In this way, one can produce examples of IREs from invariant random subgroups by taking the pushforward.
\end{example}

\begin{example}[Percolation] \label{ex:perc}
    Given an invariant percolation process on $\Gamma$ (such as Bernoulli bond percolation), we may turn it into an IRE by associating to each configuration the equivalence relation of being in the same connected component and pushing forward the distribution of the percolation process. 
\end{example}

\begin{example}\label{ex:subrel}
    Fix a pmp action $\Gamma\acts (X,\mu)$ and a Borel subequivalence relation $\cal{S}\subseteq\cal R_{\Gamma\acts X}$. This naturally produces an IRE by first pulling back $\cal S|_{[x]_{\cal R}\times[x]_ {\cal R}}$ through the map $g \mapsto g.x$ to an equivalence relation $S_x$ on $\Gamma$, and then pushing forward the measure $\mu$ through $x\mapsto S_x$. In fact, every IRE arises in this fashion (see Proposition 8.3 of \cite{TDthesis} and Proposition 15.3 of \cite{KechrisSpaceOfEqRels}).
\end{example}

In the following proposition we relate Benjamini--Schramm convergence of graphs to weak convergence of IREs, so we opt for a random variable approach. Let $\cal C$ be a Cayley graph of $\Gamma$. If $R$ is an IRE, we let $\cal C_{\textrm{con}} (o, R)$ denote the connected component of the identity $e$ in the subgraph of $\cal C$ induced by $[e]_R$ and rooted at $o=e$. It is easy to see that $\cal C_{\textrm{con}} (o,R)$ is unimodular when $R$ is an IRE.

\begin{proposition}\label{prop:conv}
    Let $\Gamma$ be a finitely generated group with bounded degree Cayley graph $\cal C$ and $R_n$ a sequence of IREs on $\Gamma$ converging in distribution to an IRE $R$. Then the unimodular random graphs $\cal C _{\textrm{con}}(o,R_n)$ Benjamini--Schramm converges to $\cal C _{\textrm{con}}(o,R)$.
\end{proposition}

\begin{proof}

Denote $(G_n,o) \coloneqq \cal C_{\textrm{con}} (o,R_n)$ and $(G,o) \coloneqq \cal C_{\textrm{con}} (o,R)$ in order to ease notation. Since all the unimodular random graphs considered are supported on induced subgraphs of $\cal C$, it suffices to show that 
\[
    \PP\big[B_r (G_n,o) \cong (H,e)\big] \longrightarrow  \PP\big[B_r (G,o) \cong (H,e)\big]
\]
for every $r\geq 0$ and every finite connected induced subgraph $H\subseteq \cal C$ containing the identity.

Fix such an $r$ and subgraph $H$ and let $\cal F$ be the family of subsets $F\subseteq B_r (\cal C, e)$ such that $(\cal C [F], e) \cong (H,e)$ where $\cal C[F]$ is the subgraph of $\mathcal C$ induced by $F$. Then
\[
     \PP[B_r (G_n,o) \cong (H,e)] = \PP[\bigcup_{F\in \cal F} \{\text{conn.~comp.~of}\ o\ \text{in}\ \cal C[[e]_{R_n} \cap B_r (\cal C,e)]\ \text{is}\ F\}],
\]
    and similarly for $(G,o)$. The evaluated event is a finite union of clopen sets, so is clopen in $\cal E$. The conclusion follows.
\end{proof}

\subsection{Finite and amenable IREs}
Let us now formally introduce the main objects of interest in this paper. 

\begin{definition}\label{def:amenIREs}
    A \textbf{finite invariant random equivalence relation (FIRE)} is an IRE $\rho$ such that $\rho$-almost every $r\in \cal E$ has all classes finite.
    
    An \textbf{amenable invariant random equivalence relation (AIRE)} is an IRE $\rho$ for which there exist Borel maps $\lambda_n\colon \cal E \to \Prob (\Gamma)$ such that for $\rho$-almost every $r\in \cal E$ and for every $g\in [e]_r$,
    \[
    \|g^{-1}.\lambda_n^r - \lambda_n^{g^{-1}.r}\|_1 \longrightarrow 0.
    \] 
\end{definition}
Denote by $\FIRE_\Gamma$ (resp. $\AIRE_\Gamma$) the subset of $\IRE_\Gamma$ consisting of finite (resp. amenable) IREs. 

\begin{remark}
    Applying this definition to the full IRE $\delta_{\Gamma\times\Gamma}$ gives precisely the definition of amenability of $\Gamma$ in terms of Hulanicki-Reiter sequences. In the other direction, every IRE on an amenable group is amenable.
\end{remark}

These definitions originally arose from studying the restricted rerooting relation.

\begin{definition}
The \textbf{restricted rerooting relation} is a Borel subequivalence relation $\Rre \subseteq \cal R_{\Gamma \actson \cal E}$. It is defined by
\[
\Rre \coloneqq \{(r,g^{-1}.r)\in \cal E \times \cal E : g\in [e]_r\}.
\] 
\end{definition}

Call an IRE $\rho$ is \textit{free} if the action $\Gamma \actson (\cal E,\rho)$ is $\rho$-essentially free. A free IRE $\rho$ is finite (resp. amenable) if and only if $\Rre$ is finite $\rho$-almost everywhere (resp. $\rho$-amenable).

\begin{remark}\label{rmk:groupoid}
In general, for a nonfree IRE the restricted rerooting relation may be amenable while the IRE is not amenable, e.g.,  the full IRE $\delta_{\Gamma\times\Gamma}$ on a nonamenable group. One can maintain a meaningful connection to amenability by instead considering the \textit{restricted rerooting groupoid} $\cal G_{re}\coloneqq \{(g,r)\in \Gamma \times \cal E : g\in [e]_r\}$ which inherits its groupoid structure as a subgroupoid of the action groupoid $\Gamma\ltimes \cal E$. Now, an IRE $\rho$ is amenable if and only if $\cal G_{re}$ is $\rho$-amenable as a discrete pmp groupoid. We won't make use of this characterization directly, but it motivates several propositions in this paper.
\end{remark}

In this paper, we are primarily interested in the relationship between $\overline{\FIRE}_\Gamma $ and $\AIRE_\Gamma$. We now prove the general fact that $\AIRE_\Gamma\subseteq \overline{\FIRE}_\Gamma $ for every countable group $\Gamma$.

\begin{proposition}\label{prop:AIREsarewFIREs}
    Let $\Gamma$ be a countable group. Every AIRE on $\Gamma$ is a weak limit of FIREs.
\end{proposition}

\begin{proof}
    Let $\rho$ be an amenable IRE on $\Gamma$ and let $\Gamma \actson (Z,\nu)$ be an essentially free pmp action. Consider the diagonal action $\Gamma \actson (Z\times \cal E,\nu\times \rho)$. Similarly to the restricted rerooting relation, define a cber  $\cal R \subseteq \cal R_{\Gamma \actson Z\times \cal E}$ by letting $((x,r),(y,t))\in \cal R$ exactly when there exists $g \in [e]_r$ such that $(g^{-1}.x,g^{-1}.r) = (y,t)$. 
    
    By the definition of an AIRE, there exists a sequence of Borel functions $\lambda_n: \cal E\rightarrow \Prob(\Gamma)$ such that for $\rho$-almost every $r \in \cal E$ and for every $g \in [e]_r$, 
    \[
    \|g^{-1}.\lambda_n^r - \lambda_n^{g^{-1}.r}\|_1 \longrightarrow 0.
    \] 
    Define maps $\eta_n: Z \times {\cal E} \rightarrow \Prob(\Gamma)$ by $\eta_n(x,r) = \lambda_n(r)$. Thus, for $(\nu \times \rho)$-almost every $(x,r)$ and for every $g \in \Gamma$ such that $(g.x,g.r)$ and $(y,t)$ are $\cal R$ equivalent, we have
    \[
    \|g.\eta_n(x,r) - \eta_n(g x, g r) \|_1 
    = \|g.\lambda_n^{r} - \lambda_n^{g.r} \|_1 
    \longrightarrow 0.
    \]
    This is exactly the hypothesis of Lemma \ref{lem:BorelAmenableRetract} (on a suitable co-null set), which we apply to see that $\cal R$ is $(\nu\times \rho)$-amenable.
    
    By \cite{CFW}, we may find an increasing sequence  $(\cal F_n)_{n\in\NN}$ of finite Borel subequivalence relations of $\cal R$ such that $\cal R = \bigcup_n \cal F_n$ up to a null set. Define $\sigma_n \colon Z\times \cal E \to \cal E$ by letting $(g,h)\in \sigma_n (x,r)$ if $(g^{-1}.(x,r), h^{-1}.(x,r))\in \cal F_n$ for $g,h \in \Gamma$ and $(x,r)\in Z\times \cal E$. One can verify that the $\sigma_n$ are Borel, $\Gamma$-equivariant, and that $(\sigma_n)_*(\nu\times \rho)$ is a FIRE. Here, we explicitly use the essential freeness of $\Gamma \acts Z$, since otherwise classes may not be finite.

    Since $F_n$ increases to $\cal R$ on a $(\nu \times \rho)$-conull set, for almost every $(x,r)\in Z\times \cal E$ we have that $\sigma_n (x,r) \longrightarrow r$ in $\cal E$. From almost sure convergence, one has convergence in distribution and thus, as IREs, $(\sigma_n)_*(\nu\times \rho) \longrightarrow P_\ast(\nu\times \rho)$, where $P \colon Z \times \cal E \to \cal E$ denotes the projection to the $\cal E$-coordinate. 
    This concludes the proof as $P_\ast(\nu\times\rho)=\rho$.
\end{proof}

\subsection{Coinduced IREs}

Coinduction is a method for producing actions of a larger group from an action of a subgroup. In the special case where the action being coinduced is an IRE, it is possible to view the action as an IRE of the ambient group. We make use of this construction later to reduce a question about general countable groups to finitely generated ones.

Given a subgroup $\Lambda \leq \Gamma$ and $\rho$ an IRE on $\Lambda$, define the \textbf{coinduced IRE} $\mathrm{CInd}_\Lambda^\Gamma(\rho)$ on $\Gamma$ in the following fashion. 
Fix a section $\sigma:\Gamma/\Lambda\to \Gamma$. 
Consider the measure space $(\mathcal E_\Lambda^{\Gamma/\Lambda}, \rho^{\Gamma/\Lambda})$ and define a map $\Phi^\sigma: \mathcal E_\Lambda^{\Gamma/\Lambda}\to \mathcal E_\Gamma$ given by $(g,h)\in \Phi^\sigma(\{r_k\}_{k\in \Gamma/\Lambda})$ if and only if $g\Lambda=h\Lambda$ and $(\sigma(g\Lambda)^{-1}g,\sigma(g\Lambda)^{-1}h)\in r_{g\Lambda}$. We then set $\mathrm{CInd}_\Lambda^\Gamma(\rho)=\Phi^\sigma_*\rho^{\Gamma/\Lambda}$. 

We compile a few properties of coinduction that we will use in the proof of Theorem \ref{thm:3}.

\begin{proposition}\label{prop:coinduction}
    Let $\Lambda \leq \Gamma$ be a subgroup and $\rho$ an IRE on $\Lambda$. Then:
\begin{enumerate}
    \item $\mathrm{CInd}_\Lambda^\Gamma(\rho)$ does not depend on the choice of section $\sigma: \Gamma/\Lambda\to \Gamma$.
    \item $\mathrm{CInd}_\Lambda^\Gamma(\rho)$ is an IRE on $\Gamma$.
    \item $\rho$ is finite if and only if $\mathrm{CInd}_\Lambda^\Gamma (\rho)$ is finite.
    \item $\rho$ is amenable implies $\mathrm{CInd}_\Lambda^\Gamma (\rho)$ is amenable.
    \item The map $\mathrm{CInd}_\Lambda^\Gamma$ is continuous with respect to the weak topology.
\end{enumerate}
\end{proposition}
\begin{proof}
(1) Fix two sections $\sigma,\tau:\Gamma/\Lambda\to \Gamma$. We may define an isomorphism $\Psi: \mathcal E_\Lambda^{\Gamma/\Lambda}\to \mathcal E_\Lambda^{\Gamma/\Lambda}$ given by $\Psi((r_k)_{k\in \Gamma/\Lambda})=(\sigma(k)^{-1}\tau(k).r_k)_{k\in \Gamma/\Lambda}$. Observe that $\Psi_*\rho^{\Gamma/\Lambda}=\rho^{\Gamma/\Lambda}$ and $\Phi^{\sigma}\circ \Psi=\Phi^{\tau}$ and so the definition of $\mathrm{CInd}_\Lambda^\Gamma(\rho)$ is independent of the choice of section.

(2) To show that $\mathrm{CInd}_\Lambda^\Gamma(\rho)$ is an IRE, fix any section $\sigma: \Gamma/\Lambda\to \Gamma$ and consider the action $\Gamma\acts \mathcal E_\Lambda^{\Gamma/\Lambda}$ given by $(g.\bold r)_k=\phi(k)^{-1}g\phi(g^{-1}k).r_{g^{-1}k}$ for $\bold r\in \mathcal E_\Lambda^{\Gamma/\Lambda}$. Observe that $\rho^{\Gamma/\Lambda}$ is invariant under this action and $\Phi\coloneqq\Phi^\sigma$ is now a $\Gamma$-equivariant map to the shift on $\mathcal E_\Gamma$.

Invariance of $\rho^{\Gamma/\Lambda}$ is clear since $g
\in \Gamma$ permutes coordinates and acts on each coordinate by an element of $\Lambda$. 
As for equivariance, $(h,h')\in g.\Phi(\bold r)$ if and only if $g^{-1}h\Lambda=g^{-1}h'\Lambda$ and $(\sigma(g^{-1}h\Lambda)^{-1}g^{-1}h,\sigma(g^{-1}h\Lambda)^{-1}g^{-1}h')\in r_{g^{-1}h\Lambda}$. On the other hand, $(h,h')\in \Phi(g.\bold r)$ if and only if $h\Lambda=h'\Lambda$ and $(\sigma(h\Lambda)^{-1}h, \sigma(h\Lambda)^{-1}h')\in (g.\bold r)_{h\Lambda}=\sigma(h\Lambda)^{-1}g\sigma(g^{-1}h\Lambda).r_{g^{-1}h\Lambda}$. Therefore, $g.\Phi(\bold r)=\Phi(g.\bold r)$ and so $\mathrm{CInd}_\Lambda^\Gamma(\rho)=\Phi_*\rho^{\Gamma/\Lambda}$ is an IRE.

(3) For an IRE on a countable group, all classes finite almost surely is equivalent to the class of the identity is finite almost surely. It is clear that the distribution of the size of the class of the identity with respect to $\rho$ is the same as the distribution of the size of the class of the identity with respect to $\mathrm{CInd}_\Lambda^\Gamma(\rho)$. Therefore one is a finite IRE if and only if the other is.

(4) Suppose $\rho$ is amenable witnessed by the sequence $\lambda_n: \mathcal E_\Lambda\to \Prob(\Lambda)\subseteq \Prob(\Gamma)$. Let $P: \mathcal E_\Gamma\to \mathcal E_\Lambda$ be the restriction to $\Lambda\times \Lambda$ and observe that $P_*\mathrm{CInd}_\Lambda^\Gamma(\rho)=\rho$.
Define $\widetilde \lambda_n: \mathcal E_\Gamma\to \Prob(\Lambda)\subseteq \Prob(\Gamma)$ by $\widetilde \lambda_n^r=\lambda_n^{Pr}$. For $\mathrm{CInd}_\Lambda^\Gamma(\rho)$-almost every $r$ and any $g\in [e]_r$, we have $g\in \Lambda$ and $P(g^{-1}.r)=g^{-1}.(Pr)$, hence
\[
\|g^{-1}.\widetilde{\lambda}_n^r-\widetilde{\lambda}_n^{g^{-1}.r}\|_1=\|g^{-1}.\lambda_n^{Pr}-\lambda_n^{g^{-1}.Pr}\|_1\longrightarrow 0
\]
the maps $\widetilde{\lambda}_n$ witness amenability of $\mathrm{CInd}(\rho)$.

(5) For continuity, it suffices to test on cylinder sets $A\subseteq \cal E_\Gamma$. Fix such an $A$ determined by a finite set $S\subseteq \Gamma$ and an equivalence relation $E$ on $S$. 
If $E$ identifies some $g,h\in S$ with $g\Lambda\not=h\Lambda$, then $\mathrm{CInd}_\Lambda^\Gamma(\mu)(A)=0$ for every $\mu$, so assume otherwise. 

Fix a section $\sigma: \Gamma/\Lambda\to \Gamma$. Let $K=\{k\in \Gamma/\Lambda: \phi(k)\Lambda\cap S\not=\varnothing\}$ and let $S_k=\sigma(k)^{-1}(\sigma(k)\Lambda\cap S)\subseteq \Lambda$ for $k\in K$. 
For each $k\in K$, define the clopen set $A_k\coloneqq\{r\in \cal E_\Lambda: r\restriction_{S_k\times S_k}=(\phi(k)^{-1}.E)\restriction_{S_k\times S_k}\}$
Then $(\Phi^\sigma)^{-1}(A)=\Pi_{k\in K} A_k\times \Pi_{k\not \in K}\cal E_\Lambda$, so for any probability measure $\mu$ on $\cal E_\Lambda$,
\[
\mathrm{CInd}_\Lambda^\Gamma(\mu)(A)=\mu^{\Gamma/\Lambda}((\Phi^\sigma)^{-1}(A))=\Pi_{k\in K}\mu(A_k).
\]
Since $\rho_n\to \rho$ weakly, we have $\rho_n(A_k)\to \rho(A_k)$ for each $k\in K$, and since $K$ is finite the products converge $\mathrm{CInd}_\Lambda^\Gamma(\rho_n)(A)=\Pi_{k\in K}\rho_n(A_k)\longrightarrow\Pi_{k\in K}\rho(A_k)= \mathrm{CInd}_\Lambda^\Gamma(\rho)(A)$.
\end{proof}

\section{IREs on exact groups}\label{sec:positive}

Proposition \ref{prop:AIREsarewFIREs} showed that AIREs arise as weak limits of FIREs in general. The goal of this section is to prove the reverse containment $\overline{\FIRE}_\Gamma\subseteq \AIRE_\Gamma$ for exact groups (Theorem \ref{thm:exactWFIREisAIRE}). We first give a sketch of the proof.

Suppose the group $\Gamma$ is exact, and let $(\rho_n)_n$ be a sequence of FIREs weakly converging to an IRE $\rho$. Fix a topologically amenable action $\Gamma \acts X$ on a compact metric space. We use finiteness to produce a coupling of each $\rho_n$ with an invariant random $X$-marking of $\Gamma$ which has the property that the mark of any element of $\Gamma$ determines equivariantly the mark of any other restricted rerooting related element. We call this property ``classwise coherence.'' Any subsequential weak limit of these markings is still a classwise coherent invariant random $X$-marking for $\rho$. Finally, we use this marking to pull back amenability of the action $\Gamma \acts X$ to $\rho$ itself.

In the definition below, let $\Gamma\acts \Xi$ be a Borel action on a standard Borel space. We consider the action on $\Xi^\Gamma$ with the usual left shift action, that is,
\[
    (g. f)(x) = f(g^{-1}x) \text{ for all } f \in \Xi^\Gamma \text{ and } g, x \in \Gamma.
\]
Additionally, we let $\Gamma$ act on $\cal{E} \times \Xi^\Gamma$ diagonally.

\begin{definition}
    Let $\Gamma\acts \Xi$ be a Borel action on a standard Borel space $\Xi$ and let $\rho$ be an IRE on $\Gamma$. A \textbf{classwise coherent} $\Xi$-marking of $\rho$ is a $\Gamma$-invariant Borel probability measure $\widetilde{\rho}$ on  ${\cal E} \times \Xi^\Gamma$ such that
    \begin{itemize}
        \item the projection of $\widetilde{\rho}$ to $\cal{E}$ is $\rho$, and
    \item  for $\widetilde{\rho}$ almost every $(r, f)$ and $(g,h) \in r$, we have $g.[f(g)] = h.[f(h)]$. 
    \end{itemize}
\end{definition}

 Note that if $\widetilde{\rho}$ is a classwise coherent marking of $\rho$, then the mark of any $g \in \Gamma$ determines the mark at any other related element $h\in \Gamma$ as $f(h) = h^{-1}g.[f(g)]$.

\begin{example}\label{example:firemarking}
    Let $\Gamma\actson \Xi$ be a Borel action on a standard Borel space $\Xi$ and let $\rho$ be a FIRE on $\Gamma$. Fix a point $x_0$ in $\Xi$. We construct a classwise coherent $\Xi$-marking of $\rho$ by marking a uniformly random point in each class by $x_0$, and then extending the marking to all other points following the requirements of classwise coherence. Explicitly, we can construct this measure as follows. Consider the (partially defined) map ${\Phi : \cal{E} \times [0,1]^\Gamma \to \cal{E} \times \Xi^\Gamma}$ given by $\Phi(r, \omega) = (r, \phi_{r,\omega})$, where
    \[
        \phi_{r,\omega}(g) = g^{-1}v_{g,r,\omega}.x_0,
    \]
    and $v_{g,r,\omega}$ is the element of $[g]_r$ whose label (with respect to $\omega$) is highest. Note that $\Phi$ is defined $\rho \otimes \texttt{Leb}^\Gamma$ almost surely. Then the desired classwise coherent $\Xi$-marking is the pushforward measure $\Phi_*(\rho \otimes \texttt{Leb}^\Gamma)$.
\end{example}

The following proposition shows that an IRE that admits a classwise coherent marking by a topologically amenable action is itself amenable.

\begin{proposition}\label{prop:amenabilitycriterion}
Let $\Gamma \actson X$ be a topologically amenable action on a Polish space $X$ and let $\rho$ be an IRE on $\Gamma$. Assume $\widetilde{\rho}$ is a classwise coherent $X$-marking of $\rho$. Then:

\begin{enumerate}
    \item There exist Borel maps $\xi_n\colon {\cal E} \times X^\Gamma \to \Prob (\Gamma)$ such that for $\widetilde{\rho}$-almost every $(r,f)$ and for every $g \in [e]_r$,
    \[
    \|g^{-1}.\xi_n^{(r,f)} - \xi_n^{g^{-1}.(r,f)}\|_1 \longrightarrow 0, \text{ and}
    \] 
    \item $\rho$ is an AIRE.
\end{enumerate}
\end{proposition}

\begin{proof}
     For part (1), by definition of topological amenability there exists a sequence of continuous maps $\eta_n \colon X\to \Prob (\Gamma)$ such that for each $g\in \Gamma$,
     \[
    \sup_{x\in X} \|g.\eta_n^x - \eta_n^{g.x} \|_1 \longrightarrow 0.
    \]

    We define $\xi_n \colon {\cal E}
    \times X^\Gamma \to \Prob (\cal E)$ by $\xi_n^{(r,f)} \coloneqq \eta_n^{f(e)}$.
    For every $g \in [e]_r$, 
    \[
        \|g^{-1} . \xi_n^{(r,f)} -  \xi_n^{g^{-1}.(r, f)}\|_1 = \|g^{-1}.\eta_n^{f(e)} - \eta_n^{f(g)} \|_1
        = \|g^{-1}.\eta_n^{f(e)} - \eta_n^{g^{-1}.[f(e)]} \|_1 \longrightarrow 0
    \]
    where we use classwise coherence for $e.[f(e)] = g.[f(g)]$.

    Part (2) follows from (1) by a standard argument as in, for instance \cite[Proposition 2.5 (ii)]{JKL} or \cite[Proposition 17.9]{KechrisSpaceOfEqRels}, which we include here. 
    
    Take the measure disintegration $(\rho_r)_{r\in\mathcal{E}}$ of $\widetilde{\rho}$ over the projection $P : \cal{E} \times X^\Gamma \to \cal{E}$, so that $\widetilde{\rho} = \int_\cal E \rho_r d\rho (r)$ where $\rho_r \in \Prob (\cal E \times X^\Gamma)$ concentrates on $P^{-1}(r)$. By uniqueness of disintegration and invariance of $\widetilde{\rho}$ and $\rho$, we have that $
    g_\ast \rho_r = \rho_{g.r}$ for every $g\in \Gamma$ and $\rho$-almost every $r \in \cal{E}$.
    
    For each $r\in \cal E$ and $n\in \NN$, define $\alpha_n^r \coloneqq \int_{\cal E \times X^\Gamma} \xi_n^{(t,f)} d\rho_r (t,f)\in \Prob(\Gamma)$.
    Observe that for almost every $r\in \cal E$, for every $g\in [e]_r$ we have $$
    g^{-1}.\alpha_n^r - \alpha_n^{g^{-1}.r} = \int_{\cal E \times X^\Gamma} g^{-1}.\xi_n^{(t,f)} - \xi_n^{g^{-1}.(t,f)} d\rho_r (t,f).
    $$
    It follows from the dominated convergence theorem that for $\rho$-almost every $r\in \cal E$ and every $g\in [e]_r$, we have $\|g^{-1}.\alpha_n^r - \alpha_n^{g^{-1}.r}\|_1 \longrightarrow 0$.
\end{proof}

\begin{remark}
    In light of Remark \ref{rmk:groupoid}, one may think of part (1) of Proposition \ref{prop:amenabilitycriterion} as proving Borel amenability of a certain extension of the restricted rerooting groupoid. The proof of part (2) of the proposition then shows that if a pmp extension of a groupoid is amenable, then the groupoid itself must be amenable. 
\end{remark}

We now have all the tools necessary for the proof of the main theorem of this section.

\begin{theorem}\label{thm:exactWFIREisAIRE}
    Let $\Gamma$ be an exact countable group. Then $\overline{\FIRE}_\Gamma  = \AIRE_\Gamma$.
\end{theorem}
\begin{proof}
    We have that $\overline{\FIRE}_\Gamma  \supseteq \AIRE_\Gamma$ by Proposition \ref{prop:AIREsarewFIREs}, so we must show that weak limits of FIREs are AIREs.
    
    Let $\Gamma$ be an exact group and $\rho_n$ a sequence of FIREs weakly converging to $\rho$. We fix a topologically amenable action $\Gamma \acts X$ on a compact metric space. We will construct a classwise coherent $X$-marking of $\rho$. 

    Let $\widetilde{\rho_n}$ be the classwise coherent $X$-marking constructed from $\rho_n$ as in Example \ref{example:firemarking}. 
    Since $\Prob(\cal E \times X^\Gamma)$ is sequentially compact, there is a (subsequential) weak limit $\widetilde{\rho}$. Then $\widetilde{\rho}$ is a classwise coherent $X$-marking of $\rho$: it projects onto $\rho$ by weak convergence and continuity of the projection map. Classwise coherence follows from the fact that the set of classwise coherent $X$-markings of $\rho$ is closed. By Proposition \ref{prop:amenabilitycriterion}, $\rho$ is an AIRE.
\end{proof}

\section{Nonexact groups and weak limits of FIREs}\label{sec:counter}

The aim of this section is to prove a strong converse to Theorem \ref{thm:exactWFIREisAIRE} by showing that it fails in every nonexact group. 

\begin{theorem}\label{thm:3}
Let $\Gamma$ be a nonexact countable group. Then there exists a weakly convergent sequence of FIREs in $\Gamma$ whose limit is not an AIRE.
\end{theorem}

Let us briefly sketch the main idea of the proof and outline the section. A (finitely generated) nonexact group admits a sequence of small scale expanders. We construct finite IREs such that the cell of the identity is (up to an $\eps$ error) one of these small scale expanders with probability at least $\eps$ (see Section \ref{sec:tessel}). In Section \ref{sec:tessel2}, we show that small scale expanders are uniformly far from hyperfinite. This then implies, roughly speaking, that the Benjamini--Schramm limit of the  connected component of the identity cannot be a hyperfinite unimodular random graph, which in turn implies the limit IRE is not amenable.

\subsection{IREs from a sequence of finite subsets}\label{sec:tessel}

We will now demonstrate that given a finite subset $A\subseteq \Gamma$ we may find a FIRE such that the cell at the identity is a translate of $A$ up to a ``small'' error with uniformly positive probability.

Denote by $|o|_R$ the size of the class $[e]_R$.

\begin{proposition}\label{thm:1}
    Let $\Gamma$ be a countable group and $(A_n)_{n\in\NN}$ a sequence of finite subsets of $\Gamma$. Then, for every $\eps > 0$ there exists a sequence of FIREs $(R_n)_{n\in\NN}$ such that
    \[
   \inf_{n\in \NN} \PP\big[|o|_{ R_n}\geq (1-\eps)| A_n|\big] > 0
    \]
    and such that $[o]_{R_n}\subseteq gA_n$ for some $g\in \Gamma$ almost surely.
\end{proposition}

The proof relies on the following construction of a random partial tiling $R$ of $\Gamma$ associated to a finite subset $A\subseteq \Gamma$ which was suggested to us by Miko\l aj Frączyk after Tom Hutchcroft used it to show that Gromov monsters admit a sequence of finite unimodular random subgraphs whose limit is nonamenable. 

Fix a finite subset $e\in A\subseteq \Gamma$ and a parameter $\delta>0$. 
Let $\Pi$ be a $\Gamma$-invariant Bernoulli random subset of $\Gamma$  with an independent uniform $[0,1]$ mark on each element of $\Pi$ such that $\PP [e\in \Pi] = \frac{\delta}{|A|}$.

The FIRE $R$ is defined by the following sampling process. First, sample from $\Pi$. For each $x\in \Pi$, place an $A$-translate $x A$. We say that a vertex $v\in \Gamma$ is \textit{conflictive} if there exist distinct $x,y\in \Pi$ that $v\in x A \cap yA$. 
The random equivalence relation $R$ is then generated by:
\begin{itemize}
\item All vertices in the complement of $\Pi A$ belong to singleton classes.

\item If $v \in \Pi A$ is not conflictive, then there exists a unique $x_v\in \Pi$ such that $v \in x_vA$, and we let $(v,x_v)\in R$.

\item If $v$ is conflictive and not in $\Pi$, let $x_v \in \Pi$ be the vertex with the least $[0,1]$ mark such that $v\in x_v A$, and set $(v,x_v)\in R$.
\end{itemize}

\noindent
Equivalently, every $g\in \Pi A-\Pi$ chooses $h_g\in \Pi$ with the smallest $[0,1]$ mark such that $g\in h_g A$. The remaining $g\in \Gamma$ choose themselves $h_g\coloneqq g$. The random relation $R$ is then given by the color classes of $g\mapsto h_g$.

The random equivalence relation $R$ clearly has every class contained inside $xA$ for some $x\in \Gamma$ and is invariant as it is constructed as a factor of the full Bernoulli shift $\Gamma\acts([0,1]^\Gamma, \texttt{Leb}^{\otimes \Gamma})$.

In Lemma \ref{lem:tesel}, we collect some estimates for $R$. Let $o\in [\Pi]_R$ denote the event where the identity $e$ is $R$-related to a vertex in $\Pi$ and let $o\not\in [\Pi]_R$ denote the negation of that event. 

\begin{lemma}\label{lem:tesel}
    The following inequalities hold for $R$.

    \begin{enumerate}[(i)]

    \item  $
    \PP \big[o\in [\Pi]_R\big] \geq  \delta - \delta^2 
    $,

    \item   $
   \expect[\Big]{{|o|_{R}} | o\in [\Pi]_R} \geq |A| (1 - 2\delta)
   $, and

   \item  $
    \PP\big[\frac{|o|_{R}}{|A|} \geq (1 - 2\delta)^2 \;\big| \;o\in [\Pi]_R\big] \geq 4\delta^2 (1 - 2\delta)^2
    $.
    \end{enumerate}
\end{lemma}

\begin{proof}


    (i) The Mass Transport Principle implies that \begin{equation}\label{eq:mtp1}
        \PP\big[o\in [\Pi]_R\big]=\EE\Big[|o|_{R}\one_{e\in \Pi}\Big].
    \end{equation}
    Writing \begin{equation}\label{eq:class.size}
        |o|_{R} = |A| - \sum_{g\in A}  \one_{ (g, e)\not\in R},
    \end{equation}
    we can combine Equations (\ref{eq:mtp1}) and (\ref{eq:class.size}) to get, by linearity of expectation, that \begin{align}
    \PP\big[o\in [\Pi]_R\big]&= |A|\PP\big[e\in\Pi\big] - \EE\bigg[\sum_{g\in A}  \one_{ (g, e)\not\in R} \one_{e\in \Pi}\bigg]\notag\\
    &= \delta - \sum_{g\in A}  \PP\big[(g,e)\not\in R \ \text{and}\  e\in \Pi\big]. \label{eq:5.3.first}
    \end{align}
    
    To compute a lower bound for $\PP\big[o\in [\Pi]_R\big]$, we are then left with finding upper bounds for each $\PP\big[(g,e)\not\in R \ \text{and}\  e\in \Pi\big]$ with $g \in A$. For each $g\in A$,
    \begin{equation}
    \PP \big[(g, e)\not\in R \ \text{and}\  e\in \Pi\big]\leq \PP \big[g\ \text{is conflictive and}\ e\in\Pi\big].\label{eq:1}
    \end{equation}
    For $g\in A$ to be conflictive, the set $g A^{-1}-\{e\}$ must contain some point of $\Pi$. Therefore, because $\Pi$ is Bernoulli random, for $g\in A$ we deduce
    \begin{align}
        \PP\big[g \ \text{is conflictive and}\ & e\in \Pi\big] = \frac{\delta}{|A|}\bigg[ 1 - \big(1 - \frac{\delta}{|A|} \big)^{|A| -1}\bigg] \label{eq:2}\\
        &\leq \frac{\delta}{|A|}(|A|-1) \frac{\delta}{|A|}  \tag{Bernoulli inequality} \leq \frac{\delta^2}{|A|}. \notag
    \end{align}
    The combination of the above leads to the conclusion
    $$
     \PP\big[o\in [\Pi]_R\big] \geq \delta - \delta^2.
    $$

    (ii) Whenever $o\in[ \Pi]_R$, the identity is $R$-related to a unique element of $\Pi$ which we denote by $x_o$. Using linearity of expectation along the lines of Equation (\ref{eq:class.size}), we deduce that $$
    \EE\Big[|o|_{R} \one_{o \in [\Pi]_R}\Big]  =  |A|\PP\big[o\in [\Pi]_R\big] - \sum_{g\in A} \EE\Big[\one_{(x_o g,x_o)\not\in R}\one_{o\in[\Pi]_R}\Big],
    $$
    and so by (i) we get \begin{equation}\label{eq:2.1}
        \EE\Big[|o|_{R} \one_{o \in [\Pi]_R}\Big]  \geq |A| (\delta - \delta^2) -\sum_{g\in A} \EE\Big[\one_{(x_o g,x_o)\not\in R}\one_{o\in[\Pi]_R}\Big].
    \end{equation}

    The Mass Transport Principle implies that \begin{equation}\label{eq:2.2}
    \EE\Big[\one_{(x_o g,x_o)\not\in R}\one_{o\in[\Pi]_R}\Big] = \EE\Big[|o|_R\one_{(g,e)\not\in R}\one_{ e\in \Pi}\Big] \leq |A| \PP\big[(g,e)\not\in R \ \text{and}\ e\in \Pi\big].
    \end{equation}
   
    Since $|o|_R \leq |A|$ almost surely.  Equations (\ref{eq:1}) and (\ref{eq:2}) imply that the right hand side of Equation (\ref{eq:2.2}) is bounded above by $\delta^2$. Combining this bound with Equation (\ref{eq:2.1}) we deduce that $$
    \EE\Big[|o|_R\one_{o\in [\Pi]_R}\Big]\geq|A|\left(\delta - \delta^2-\delta^2 \right).
    $$
    
    To conclude the proof, observe that Equation (\ref{eq:5.3.first}) implies that $\PP\big[o\in[\Pi]_R\big] \leq \delta$, so
    \[
         \EE\Big[|o|_{R} \;\big|\; o\in [\Pi]_R\Big] \geq \frac{1}{\delta}\EE\Big[|o|_{R} \one_{o\in[\Pi]_ R}\Big]\geq |A|(1-2\delta).
    \]

    (iii) This is a straightforward application of the Paley--Zygmund inequality to the random variable $\frac{|o|_R}{|A|}$ conditioned on $o\in[\Pi]_R$, using part (ii) and ${\expect[\big]{{|o|_{R}^2} | o\in [\Pi]_R}\leq |A|^2}$.
\end{proof}

\begin{proof}[Proof of Proposition \ref{thm:1}]
    Let $\delta>0$ be small enough so that $$
    (1 - 2\delta)^2\geq 1 - \eps.
    $$
    For each $n\in \NN$, let $R_n$ be the FIRE associated to $A_n$ and $\delta$, as constructed in the preamble of the previous lemma. Conclusions (i) and (iii) of Lemma \ref{lem:tesel} imply that $$
    \PP\big[|o|_{R_n}\geq (1-\eps)|A_n|\big] \geq (\delta - \delta^2)4\delta^2(1 - 2\delta)^2 >0,
    $$
    proving the proposition.  
\end{proof}

\subsection{Small scale expanders are uniformly far from hyperfinite}\label{sec:tessel2}

Small scale expander graphs are non-hyperfinite in a robust way: if one deletes a sufficiently small proportion of the vertices from a small scale expander, the resulting graph cannot be hyperfinite. More precisely:

\begin{proposition}\label{thm:2}
    Let  $G$ be a small scale $(\kappa,N)$-expander with degree at most $d\in \NN$. Let $ 0 < \eps < \frac{\kappa}{2 (1 + d) + \kappa}$. Then any induced subgraph $G[A] \subseteq G$ with $|A|\geq (1-\eps)|V(G)|$ is not $(\eps,N)$-hyperfinite.
\end{proposition}

\begin{proof}
    We will prove the contrapositive. Suppose $G$ is a small scale $(\kappa,N)$-expander with degree at most $d$ and that there exists $A \subseteq V(G)$ with $|A| \geq (1- \eps) |V(G)|$ such that $G[A]$ is $(\eps,N)$-hyperfinite.  Hence there exists a set $E \subseteq E(G[A])$ such that $|E|\leq \eps |A|$ and $G[A]- E$ has connected components of size at most $N$.

    Let $\cal F$ be the family of connected components of $G [A]- E$, and $E(\cal F)$ denote the union of the edge sets of the graphs of $\cal F$. 
    Thus,  
    \[
         |E(G) - E(\cal F)| \leq |E| + d|V(G)-A| \leq \eps |A| + \eps d|V(G)| \leq \eps (1 + d) |V(G)|.
    \]
On the other hand, we get the following bound from small scale expansion, 
\[
       |E(G) - E(\cal F)|  \geq \frac{1}{2} \sum_{F\in \cal F} |\partial_{G} F|\geq \frac{\kappa}{2}\sum_{F\in\cal F} |F| = \frac{\kappa}{2} |A|\geq \frac{\kappa}{2}(1-\eps)|V(G)|
\]
    where the first inequality is true because the sum $\sum_{F\in \cal F} |\partial_{G} F|$ counts edges between $\cal F$ components at most twice, and edges from a $\cal F$ component to $V(G) -A$ at most once.
    
    We have shown that $\kappa (1 - \eps) \leq 2\eps (1 + d)$ and thus $ \frac{\kappa}{2 (1 + d) + \kappa} \leq \eps$, concluding the proof.
\end{proof}

\subsection{Nonamenable weak limits of FIREs}\label{sec:tessel3}

If $\cal C$ is a Cayley graph of $\Gamma$ and $R$ is an IRE, we let 
$\cal C(o,R)$ denote the (possibly disconnected) random subgraph of $\cal C$ induced by $[e]_R$ and rooted at $o=e$. We then distinguish this from $\cal C_{\textrm{con}} (o, R)$  which will denote the connected component of $o$ in $\cal C(o,R)$.

We can now prove that nonexact countable groups admit a convergent sequence of finite IREs whose limit is a nonamenable IRE.

\begin{proof}[Proof of Theorem \ref{thm:3}] 
We first reduce to the case where $\Gamma$ is finitely generated.

Since increasing unions of exact groups are exact \cite[Exercise 5.1.1]{Bro}, if $\Gamma$ is nonexact, there is a finitely generated subgroup  $\Lambda \leq \Gamma$ that is also nonexact. Suppose that $(R_n)_n$ is a sequence of finite IREs on $\Lambda$ weakly converging to a nonamenable IRE $R$ on $\Lambda$. Proposition \ref{prop:coinduction} then implies that the sequence $\mathrm{CInd}_\Lambda^\Gamma(R_n)$ of finite IREs converges weakly to the nonamenable IRE $\mathrm{CInd}_\Lambda^\Gamma(R)$. 

Now suppose $\Gamma$ is finitely generated nonexact and let $\cal C$ be a Cayley graph for $\Gamma$ whose degree we denote by $d$. Then there exists a sequence $(G_n)_n$ of induced subgraphs of $\cal C$ with $e\in V(G_n)$ that forms a sequence of small scale $\kappa$-expanders, for some fixed $\kappa > 0$. Fix $\eps>0$ where $\eps$ is small enough to satisfy the hypotheses of Proposition \ref{thm:2} and let $(R_n)_{n}$ be the sequence of FIREs obtained by applying Proposition \ref{thm:1} with inputs $\eps$ and the sequence of subsets $(V(G_n))_{n}$. By sequential compactness we may assume, upon possibly passing to a subsequence, that the sequence $(R_n)_{n}$ of FIREs weakly converges to an IRE $R$. It remains to show that $R$ is nonamenable.
 
Let us therefore suppose for the sake of contradiction that $R$ is an amenable IRE. We will find $N\in \NN$ and a subset $A\subseteq V( G_N)$ with $|A|\geq (1-\eps)|V(G)|$ such that the induced subgraph $G_N[A]$ is $(\eps,N)$-hyperfinite, contradicting Theorem \ref{thm:2}.

A coupling argument using \cite{CFW} as in Proposition \ref{prop:AIREsarewFIREs} shows that $\cal C_{\textrm{con}} (o,R)$ may be obtained by sampling connected components of a $\mu$-hyperfinite Borel graph. Hence $\cal C_{\textrm{con}} (o,R)$ is a hyperfinite unimodular random graph by Proposition \ref{prop:hyper}. Let $\delta >0$ be small enough such that $$
 \frac{\delta}{2c(1 - \eps)} \leq \eps,
 $$
 where $c\coloneqq \inf_{n\in \NN}\PP\big[|o|_{R_n}\geq (1-\eps)|V(G_n)|\big] >0$. 
 It follows from Proposition \ref{prop:conv} and Theorem \ref{prop:schramm} that there exists $k\in \NN$ such that for $n$ sufficiently large $\cal C_{\textrm{con}} (o,R_n)$ is $(\delta,k)$-hyperfinite. Hence there exists $N$ such that $\cal C_{\textrm{con}}(o,R_N)$ is $(\delta,N)$-hyperfinite. 
 Thus there exists a rerooting equivariant subset of edges $K (\cal C_{\textrm{con}} (o,R_N))\subseteq E (\cal C_{\textrm{con}} (o,R_N))$ such that $\EE \big[\deg_{K (\cal C_{\textrm{con}} (o,R_N))} (o) \big] \leq \delta$
 and $\cal C_{\textrm{con}}(o,R_N) - K (\cal C_{\textrm{con}} (o,R_N))$ has connected components of size at most $k$ almost surely. 

Define a coupling $(\cal C (o,R_N),E_N,o)$ by letting 
\[
E_N \coloneqq \bigcup_{g\in [e]_{R_N}} g K (\cal C (o,g^{-1}.R_N)) \subseteq E(\cal C (o,R_N).
\]
By construction $E_N$ coincides with $K(\cal C_{\textrm{con}}(o,R_N))$ on each connected component of $\cal C(o,R_N)$. Therefore $\EE\big[\mathrm{deg}_{E_N}(o)\big] \leq \delta$ and $\cal C (o,R_N) - E_N$ has connected components of size at most $k$ almost surely.

We can then compute \begin{align*}
    \EE\Big[\frac{|E_N|}{|o|_{R_N}}\;&\Big|\; {|o|_{R_N}}\geq (1-\eps)|V(G_N)|\Big]\\
    &\leq \frac{1}{(1-\eps)|V(G_N)|} \expect[\Big]{\frac{1}{2}\sum_{v\in  V(\cal C(o,R_N))} \mathrm{deg}_{E_N}(v) |{|o|_{R_N}\geq (1-\eps)|V(G_N)|}} \\
    &\leq\frac{1}{2 (1 - \eps)} \expect[\Big]{\mathrm{deg}_{E_N}(o) | {|o|_{R_N}\geq(1-\eps)|V(G_N)|}}
    \end{align*}
    by the conditioning assumption and mass transport. 
    A further upper bound
    \[
     \EE\Big[\frac{|E_N|}{|o|_{R_N}}\;\Big|\; {|o|_{R_N}}\geq (1-\eps)|V(G_N)|\Big]\leq\frac{1}{2 (1-\eps)} \frac{\expect[\big]{\mathrm{deg}_{E_n}(o)}}{\PP\big[{|o|_{R_N}\geq (1-\eps)|V(G_N)|}\big]} \leq \frac{\delta}{2c(1 - \eps)} \leq \eps
    \]
    is established by using the general bound $\EE[X\,|\, A] \leq \EE[X]/\PP[A]$ for non-negative random variables.
    It then follows from the first moment principle that there exists a finite subset $A \subseteq V(G_N)$ with $|A| \geq (1-\eps)|V(G_N)|$ such that $G[A]$ is $(\eps,N)$-hyperfinite, contradicting Theorem \ref{thm:2} and showing the limit IRE $R$ is not amenable. 
\end{proof}

\section{Ideal Bernoulli Voronoi tessellations and questions}\label{sec:IBVTQ}
 
This paper grew out of a desire to better understand ideal Bernoulli Voronoi tessellations, the discrete analog of ideal Poisson Voronoi tessellations. The ideal Bernoulli Voronoi tessellation was first studied in the case of regular trees by Bhupatiraju \cite{bhu}.

Fix a left-invariant proper metric $d$ on $\Gamma$ (for example, the word metric associated to a finite generating set). The \textbf{Voronoi tessellation} associated to a configuration  $\omega \subseteq \Gamma$ is the ensemble of sets $\{V_\omega(g)\}_{g \in \omega}$, where
\[
    V_\omega(g) \coloneqq \{ h \in \Gamma \mid d(h, g) \leq d(h, g') \text{ for all } g' \in \omega \}.
\]
Note that the Voronoi tessellation is equivariantly defined (as the metric is left-invariant). It is not a genuine partition however, as the cells can have boundary points in common. It is easy to refine the Voronoi tessellation to a genuine partition whilst maintaining equivariance: for example, each point on the boundary of a cell (that is, whose minimal distance to $\omega$ is achieved by multiple points of $\omega$) simply picks one at random. There are other methods to resolve the boundaries, we fix one and continue to use the above notation for the refined version.

The \textbf{Bernoulli Voronoi tessellation} $\BVT_p$ with parameter $p \in (0,1)$ is the Voronoi tessellation of the Bernoulli random subset $\Pi_p \subseteq \Gamma$ (that is, each $g \in \Gamma$ is independently in $\Pi_p$ with probability $p$). 
$\BVT_p$ is an example of a finite IRE as implied by the following lemma.

\begin{lemma}
    Let $\Pi$ be a $\Gamma$-invariant random subset of $\Gamma$ with intensity $\PP[e \in \Gamma] > 0$. Suppose $C(g, \Pi) \subseteq \Gamma$ is a family of measurably defined subsets for each $g \in \Pi$ such that
    \begin{itemize}
        \item (Equivariance) $C( hg, h \Pi) = h. C(g, \Pi)$ for all $h \in \Gamma$, and
        \item (Disjointness) $C(g, \Pi) \cap C(g', \Pi) = \empt$ for all distinct $g,g' \in \Pi$.
    \end{itemize}
    Then $\Pi$ almost surely, every cell $C(g, \Pi)$ is finite.
\end{lemma}

\begin{proof}
    We get that $\EE\big[\abs{C(e, \Pi)}\one_{e \in \Pi}\big] \leq 1$ by applying mass transport with the function $f(x,y;\Pi) = \one_{x \in C(y,\Pi)}$ . It follows that $\abs{C(e, \Pi)}\one_{e \in \Pi}$ is finite almost surely, and therefore $\PP\big[g \in \Pi \text{ and } C(g, \Pi) \text{ is infinite}\big] = 0$ for each $x \in \Gamma$ by invariance.

    Since $\Gamma$ is countable, almost surely for every $g\in \Gamma$ either $g\not \in \Pi$ or the cell $C(g,\Pi)$ is finite.
\end{proof}

For a nonamenable group $\Gamma$, any subsequential weak limit of $\BVT_p$ defines a nontrivial IRE on $\Gamma$. Indeed, the nontriviality follows from Kechris \cite[Theorem 15.7] {KechrisSpaceOfEqRels}. We call these subsequential weak limits \emph{ideal Bernoulli Voronoi tessellations}. Let us state some basic questions about them.

One step in the proof of fixed price for higher rank semisimple Lie groups \cite{IPVTHighRank} consists of proving that the ideal Poisson Voronoi tessellation is an amenable IRE. It is natural to raise the question of amenability in general.

\begin{question}
    Are ideal Bernoulli Voronoi tessellations amenable as IREs?
\end{question}

It follows from Theorem \ref{thm:exactWFIREisAIRE} that for exact groups the answer is positive. However, Theorem \ref{thm:3} shows that it is not sufficient to be a weak limit of FIREs in order to be amenable and so the general question remains open.

In the continuous setting, it is known that there is a unique limit for the  ideal Poisson Voronoi tessellation on hyperbolic space \cite{russetal} and more general symmetric spaces \cite{IPVTHighRank}. Bhupatiraju \cite{bhu} showed that the limit is unique for regular trees.

\begin{question}
    In what generality can one say that there is a \emph{unique} ideal Bernoulli Voronoi tessellation on a group? 
\end{question}

We may view $\cal{E}$ as a subset of $(2^\Gamma)^\Gamma$ by identifying each $r \in \cal{E}$ with the function $x \mapsto [x]_r$. A \textbf{cell selection rule} is a measurable and equivariant map $\phi : \cal{E} \to (2^\Gamma)^\Gamma$ such that for all $x \in \Gamma$, either $\phi_r(x) = [x]_r$ or $\phi_r(x) = \empt$. That is, $\phi$ is a measurable way of selecting some of the cells of equivalence relations. A cell selection rule is \textbf{trivial} if it selects all cells, or selects no cells. An IRE $\rho$ is \textbf{indistinguishable} it admits no nontrivial cell selection rules.
The second author \cite{indiSam} recently showed that the ideal Poisson Voronoi tessellation on symmetric spaces is indistinguishable. 

\begin{question}
    Are the ideal Bernoulli Voronoi tessellations always indistinguishable? Under what conditions is a weak limit of finite IREs indistinguishable?
\end{question}

\def\MR#1{}
\bibliographystyle{amsalpha} 
\bibliography{bib}

\end{document}